\documentclass[12pt]{amsproc}%
\usepackage[utf8]{inputenc}
\usepackage{amsmath}
\usepackage{amssymb}
\usepackage{amsthm}
\def\leq {\leqslant}
\def\var{\operatorname{Var}}
\def\le {\leqslant}
\def\ge {\geqslant}
\def\geq {\geqslant}

\usepackage{amssymb}
\usepackage{amsfonts}
\usepackage{amsmath}
\usepackage{graphicx}%
\DeclareMathOperator*{\supess}{sup\,ess}
\DeclareMathOperator*{\infess}{inf\,ess}
\RequirePackage{srcltx}
\usepackage{amssymb}

\newcommand{\norm}[1]{ \left\lVert#1\right\rVert}
\providecommand{\U}[1]{\protect\rule{.1in}{.1in}}
\theoremstyle{plain}

\newtheorem{theorem}{Theorem}[section]

\newcommand\h[1]{\mkern2mu\widehat{\mkern-2mu#1}}

\newcommand\R{\mathbb{R}}
\newcommand\N{\mathbb{N}}

\newtheorem{corollary}[theorem]{Corollary}

\theoremstyle{remark}
 \RequirePackage{srcltx}

\newtheorem{remark}[theorem]{Remark}

\numberwithin{equation}{section}

\begin{document}
\title[Note on Fourier inequalities]
{
Note on Fourier inequalities
}

\author{Miquel Saucedo}
\address{M. Saucedo,
 Centre de Recerca Matem\`{a}tica\\
Campus de Bellaterra, Edifici C
08193 Bellaterra (Barcelona), Spain;
Universitat Aut\`{o}noma de Barcelona, Campus de Bellaterra, Edifici C
08193 Bellaterra (Barcelona), Spain.}

\email{miquelsaucedo98@gmail.com}

\author{Sergey Tikhonov}
\address{S. Tikhonov,
ICREA, Pg. Llu\'{i}s Companys 23, 08010 Barcelona, Spain\\
Centre de Recerca Matem\`{a}tica\\
Campus de Bellaterra, Edifici C
08193 Bellaterra (Barcelona), Spain,
 and Universitat Aut\`{o}noma de Barcelona,
  Campus de Bellaterra, Edifici C
08193 Bellaterra (Barcelona), Spain
 .}
\email{ stikhonov@crm.cat}
\thanks{M. Saucedo is supported by  the Spanish Ministry of Universities through the FPU contract FPU21/04230.
 S. Tikhonov is supported
by PID2020-114948GB-I00, 2021 SGR 00087, AP 14870758,
the CERCA Programme of the Generalitat de Catalunya, and Severo Ochoa and Mar\'{i}a de Maeztu
Program for Centers and Units of Excellence in R\&D (CEX2020-001084-M)
}
\subjclass[2010]{Primary  42B10; Secondary
46E30, 26D10.}
\keywords{Fourier transform, variable Lebesgue space, Weighted Fourier inequalities, Hardy--Littlewood type theorem}

\begin{abstract}
We prove that
  the Hausdorff--Young
  inequality
$\|{\widehat{f}}\|_{q(\cdot)}
 \leq C \norm{f}_{p(\cdot)}$
 with
$q(x)=p'(1/x)$ and $p(\cdot)$ even and non-decreasing
  holds in  variable Lebesgue spaces if and only if $p$ is a constant.
 However, under the additional condition on monotonicity of   $f$,
we obtain a full characterization of   Pitt-type  weighted Fourier inequalities
  in the classical and variable Lebesgue setting.
\end{abstract}
\maketitle

\section{Fourier inequalities in variable $L_p$ spaces}

For $f\in L_1(\R^n)$ we define the Fourier transform by

	\[
	\widehat{f}(y)= \int_{\R^n} f(x)e^{-2\pi i(x,y)}\, dx.
	\]
An important estimate  for the Fourier transform   is the Hausdorff--Young
inequality: for $1\le p \le 2$,
  \begin{equation}
    \label{hy}
  \|{\widehat{f}}\|_{p'} \leq \|{f}\|_{p},
    \end{equation}
where $p'$ is the conjugate exponent for $p$, that is,  $\frac1p+\frac1{p'}=1$.
A natural question is to obtain an analogue of the Hausdorff--Young inequality for  the variable Lebesgue
spaces \cite[Section 5.6.10]{cruz}.

Let $f$ and $p$ be measurable functions. Assume further that
$1\le p(x) \le \infty$. We set
$\rho_p(f)=\int_{\mathbb{R}} |f(x)|^{p(x)} dx$ and define $$
\|{{f}}\|_{p(\cdot)}=\inf\Big\{\lambda>0: \rho_p(f/\lambda)\le 1\Big\}.  
$$
A typical regularity condition on $p$ is the
log-Hölder continuity, i.e, that there exists a constant
$C$ such that for all $x,  y\in \R$,  $|x -y| < 1/2$,
$|p(x)-p(y)|\le \frac{C}{-\log (|x -y|)}.$
If
$p(x)=\infty$ we assume that there is $C>0$ such that  
$\frac1{p(y)}\le \frac{C}{-\log (|x -y|)}$
for  $|x -y| < 1/2$.

In \cite[Question 5.61]{cruz}, it was proved that a natural extension of (\ref{hy}), namely,
$$
  \|{\widehat{f}}\|_{p'(\cdot)} \leq C \|{f}\|_{p(\cdot)},\qquad \frac1{p(\cdot)}+\frac1{p'(\cdot)}=1,
$$
does not hold in general and the following problem was posed:

{\it
Let $p(\cdot)$ be even and non-decreasing on $\R_+$ such that 
 \begin{equation}  \label{cond}
 1\le \infess\limits_{x\in \R} p(x)\le \supess\limits_{x\in \R} p(x) \le 2.
  \end{equation}
Is it the case that there exists a constant $C > 0$ such
that for any $ f\in L_{p(\cdot)}$
$$
  \|{\widehat{f}}\|_{q(\cdot)} \leq C \|{f}\|_{p(\cdot)}\qquad\mbox{with}\quad q(x)=p'(1/x)?
$$
}
We give
a negative answer to this question by demonstrating that such an inequality is valid  only in the classical, non-variable case.

\begin{theorem}
Let  $p$ be a  continuous even non-decreasing function on $\R_+$ such that $$\lim_{x\to \infty} p(x)=p_\diamond\le 2$$
and condition (\ref{cond}) holds. Let $q$ be log-Hölder continuous.  If  the Hausdorff--Young inequality
  \begin{equation}
    \label{hy-}
\|{\widehat{f}}\|_{q(\cdot)}
 \leq C \norm{f}_{p(\cdot)}
    \end{equation}
     holds for any  Schwartz function $f$, then $p(x)\leq p_\diamond \leq q'(x)$. In particular, if $q(x)=p'(1/x)$, then $p(x)=p_\diamond$.
\end{theorem}
\begin{remark}
\label{remark:modular}

It is worth mentioning that the following non-classical
Fourier inequality
$$\int_{\R}|{\widehat{f}}(\xi)|^{q(\xi)}d\xi
 \leq C_1
 \int_{\R}|{{f}}(x)|^{p(x)}dx+C_2
 $$
holds,
with positive constants $C_1$ and $C_2$ that depend on $p$ and $q$ but not on $f$, under some additional conditions on $p(\cdot)$ and $q(\cdot)$. See
\cite[Corollary 1.16]{cruz2}.
\end{remark}

\begin{proof}
First, since $p(x)\leq p_\diamond$ it is enough to show that $q'(x)\geq p_\diamond$ or, equivalently, $q(x)\leq p_\diamond'$. Furthermore, by the modulation property of the Fourier transform, it suffices to show the result for $x=0$, that is,
 $$
 q(0)\leq p_\diamond'.
 $$
  Second, assuming that (\ref{hy-})
 holds for  any Schwartz function,  
we  have, for any such $f$,
  \begin{equation}
    \label{hy--}\|{\widehat{f}}\|_{q(\cdot)} \leq C \norm{f}_{p_\diamond}.
     \end{equation}
Indeed,
letting 
 $\tau_h f (x)= f(x-h)$, we get $$\|{\widehat{f}}\|_{q(\cdot)}  \leq C \norm{\tau_h f}_{p(\cdot)}.$$
 To verify (\ref{hy--}), we will  show that
  \begin{equation}
    \label{hy--+}\lim_{h\to\infty} \norm{\tau_h f}_{p(\cdot)}=\norm{f}_{p_\diamond}.
     \end{equation}
Noting that
$$\rho_p\Big(\frac{\tau_h f}{\norm{f}_{p_\diamond}}\Big)=\int_{\mathbb{R}} \Big(\frac{|f(x)|}{\norm{f}_{p_\diamond}}\Big)^{p(x+h)} dx$$
and since $p$ is a bounded function,
the dominated convergence theorem yields
$$\lim_{h\to\infty} \rho_p\Big(\frac{\tau_h f}{\norm{f}_{p_\diamond}}\Big)=\int_{\mathbb{R}} \Big(\frac{|f(x)|}{\norm{f}_{p_\diamond}}\Big)^{p_\diamond}=1.$$
Thus,
$$\lim_{h\to\infty}\rho_p\Big(\frac{\tau_h f}{\norm{f}_{p_\diamond}}\Big)=1,$$ which implies (\ref{hy--+}), due to the fact that
$p(x)\ge1$ and
if $1-\varepsilon<\rho_p\Big(\frac{\tau_h f}{\norm{f}_{p_\diamond}}\Big)< 1+ \varepsilon,$ then
$\rho_p\Big(\frac{\tau_h f}{(1+\varepsilon)\norm{f}_{p_\diamond}}\Big)<1<\rho_p\Big(\frac{\tau_h f}{(1-\varepsilon)\norm{f}_{p_\diamond}}\Big)$.

Third, assuming that $q$ is log-Hölder continuous at the origin and  that inequality  (\ref{hy--}) holds, we prove that  $q(0)\leq p_\diamond'$.
For a Schwartz function $f$  satisfying $\widehat{f}(\xi) \geq 1$ for $|\xi|\leq 1$, we define $$f_\lambda(x)=\lambda f(\lambda x).$$
Then,
$\norm{f_\lambda}_{p_\diamond}= \lambda^{1/p_\diamond'} \norm{f}_{p_\diamond}.$
 We consider two cases. If $q(0)<\infty$, $$\rho_q(C\lambda^{-\frac{1}{q(0)}} \widehat{f_\lambda}) > \int_{-\lambda}^{\lambda} |C\lambda^{-\frac{1}{q(0)}} \widehat{f}(\xi/\lambda)|^{q(\xi)} d\xi>
 \lambda^{-1}\int_{-\lambda}^{\lambda} C^{q(\xi)}\lambda^{1-\frac{q(\xi)}{q(0)}} d\xi.$$
Hence, by log-Hölder continuity, for some $C$ we have
$\lim\limits _{\lambda \to 0} \rho_q(C\lambda^{-\frac{1}{q(0)}} \widehat{f}_\lambda) \geq 1.
$
Therefore, we deduce that
$$ \lim_{\lambda \to 0} \big\|\widehat{f_\lambda}\big\|_{q(\cdot)} \lambda^{-\frac{1}{q(0)}}\geq \frac1{C}
$$
and
$$
\lambda^{\frac1{p_\diamond'}} \norm{f}_{p_\diamond}
=
\big\|{f_\lambda}\big\|_{p_\diamond}
\gtrsim
\big\|\widehat{f_\lambda}\big\|_{q(\cdot)} \gtrsim \lambda^{\frac{1}{q(0)}},$$
that is, $q(0)\leq p_\diamond'$.

Finally, we assume that $q(0)=\infty$. If for any $\lambda>0$ the set $\{-\lambda<x<\lambda: q(x)=\infty\}$ has positive measure, we deduce that for any $\lambda>0$, $\rho_q(\widehat{f}_\lambda)\geq 1$ and, consequently, $
\big\|\widehat{f}_\lambda\big\|_{q(\cdot)}
\geq 1$.
Then the Hausdorff--Young inequality
      \eqref{hy-} implies that  $p_\diamond'=\infty$.

Otherwise, by the log-Hölder continuity of $q$, for large enough $C$,
we derive as above that
$$\lim _ {\lambda \to 0 }\rho_q(C\widehat{f_\lambda} ) \geq \int_{-\lambda}^{\lambda} | C|^{q(\xi)} d\xi \geq 1,
$$
whence it follows that  $p_\diamond=1$.

\end{proof}

\section{Hardy--Littlewood theorem in the variable exponent setting }
We now show that some nontrivial Fourier inequalities     
 can still be derived in the variable exponent setting  assuming additional conditions on $f$.
We obtain an analogue of the classical Hardy--Littlewood inequality (see  \cite{tit, sad}) 
\begin{equation}
\label{hl inequality}
\|\h{f}(\xi)|\xi|^{n(1-\frac2p)}\|_p\lesssim\|f\|_p, \qquad 1<p\le 2,
\end{equation}
for radial monotone functions; cf. \cite{  car, jam, alberto1}.
For a radial function $g$,  we denote by $g_0$ its radial part: 
 $g(x)=g_0(|\xi|)$.
\begin{theorem}
Let $n\in \N$ and $p(\cdot)$ be a radial function such that $p_0(\cdot)\ge 1$ is  decreasing and bounded at the origin. The inequality
    $$\int_{\mathbb{R}^n} |\h{f}(\xi) |^{p(\xi)} |\xi|^{n(p(\xi)-2)} v\big(1/{|\xi|}\big) \,d\xi
 \lesssim
 \int_{\mathbb{R}^n} |{f}(x) |^{p(1/{|x|})} v(|x|) \,dx
$$
holds for any radial non-negative function $f$  such that $f_0(x) x^{n-1}$ is non-increasing with $$\lim _{x \to 0} f_0(x) x^{n-1} \leq 1$$
if and only if, for any positive $r,$
\begin{equation}\label{2.1}
    \int_{|x|\geq r} r^{p(1/|x|)}|x|^{-np(1/|x|)} v(|x|) d x  \lesssim \int_{|x|\leq r} v(|x|) |x|^{(1-n) p(1/|x|)} dx.
\end{equation}
\end{theorem}
\begin{proof}
First, we note that under the  conditions on $f$, the Fourier transform is well defined; cf.
 Remark \ref{remark:well-def}.

To see the necessity of condition \eqref{2.1}, we
set $f(x)=\chi(|x|)_{[0,r]} |x|^{-n+1}$.
Then, for $\xi \leq \frac{\pi}{3} r^{-1}$ we have
$\widehat{f}(\xi)\gtrsim r$. Hence,
    $$ \int_{|\xi|\leq \frac{\pi}{3} r^{-1}} r^{p(\xi)}|\xi|^{n(p(\xi)-2)} v(|\xi|^{-1}) d \xi = \int_{|\xi|\geq \frac{3}{\pi}r} r^{p(\xi^{-1})}|\xi|^{-np(\xi^{-1})} v(|\xi|) d \xi  \lesssim \int_{|x|\leq r} v(x) |x|^{p(\xi^{-1})(1-n)} dx.$$

We now prove the sufficiency of condition \eqref{2.1}. Using Theorem \ref{estimate}, we have
    \begin{eqnarray*}
  J&=&  \int_{\mathbb{R}^n} |\h{f}(\xi) |^{p(\xi)} |\xi|^{n(p(\xi)-2)} v(|\xi|^{-1}) \,d\xi \\
  &\lesssim&
     \int_0 ^\infty \xi ^{n(p_0(\xi)-2)+n-1} v(\xi^{-1}) \left(\int_0 ^{\xi^{-1}} f_0(t) t^{n-1} dt  \right)^{p_0(\xi)} d \xi
  \\ &=&
 \int_0 ^\infty \xi ^{-n(p_0(1/\xi)-2)-n-1} v(\xi) \left(\int_0 ^{\xi} f_0(t) t^{n-1} dt  \right)^{p_0(1/\xi)} d \xi.
 \end{eqnarray*}
 Rewriting this  using
  the Hardy operator $H(g)(x)=\frac1x\int_0^x g(t)dt$, we obtain
 \begin{eqnarray*}
J &\lesssim& \int_0 ^\infty
\xi ^{(n-1)(1-p_0(\xi^{-1}))} v(\xi)
\Big(
H(f_0(t)t^{n-1})\Big)^{p_0(\xi^{-1})} (\xi) d \xi  \\ &\lesssim&
 \int_0 ^\infty \xi ^{(n-1)(1-p_0(\xi^{-1}))} v(\xi)
 \Big(f_0(\xi)\xi^{n-1}\Big)^{p_0(\xi^{-1})} (\xi) d \xi
  \\ &=&\int_0 ^\infty v(\xi) \xi^{n-1} f_0(\xi)^{p_0(\xi^{-1})}\asymp  \int_{\mathbb{R}^n} |{f}(x) |^{p(x^{-1})} v(|x|) \,dx.
 \end{eqnarray*}
Here the second inequality follows from the characterization of
 the boundedness of the Hardy operator for weighted modulars restricted to decreasing functions, as obtained in    \cite[Theorem 3]{neu}. It is a routine calculation to see that \eqref{2.1} is the required condition.

\end{proof}

\vspace{0.1cm}
\section{Fourier inequalities for monotone functions in the classical $L_p$}
First, we recall that for radial $f$, its Fourier transform is given by
  \begin{equation}\label{def-fourier}
    \widehat{f}(|\xi|) = (2 \pi)^{n/2} |\xi|^{1-n/2
    }
\int_{0}^{\infty}r^{n/2-1}f(r)J_{n/2-1}(|\xi|r)r\mathrm{d}r,
\end{equation}
where $J_\alpha(z)$ is the  Bessel function of the first kind of order $\alpha.$
Our goal is to characterize  weighted Pitt's
inequality restricted to radially  decreasing functions. First,
we  obtain the following upper bound for the Fourier transform of such functions: \begin{theorem}\label{estimate}
 Let $f$ be a radial function in $\R^n.$
  Assume that $r^{\frac{n-1}{2}}f_0(r)$ is a non-increasing function with limit zero. Then we have
  \begin{equation}\label{bound for fourier}|\widehat{f}(|\xi|)|\lesssim \int_0 ^{|\xi|^{-1}} r^{n-1} f_0(r) dr.
  \end{equation}
\end{theorem}

\begin{remark}
\label{remark:well-def}

\begin{itemize}
\item[($i$)] Note that under the conditions of Theorem \ref{estimate}
the Fourier transform can be defined by
\eqref{def-fourier} as
$\lim_{N\to \infty}\int_{0}^{N}$.

\item[($ii$)]
 Estimate \eqref{bound for fourier} is not valid if we only assume in Theorem \ref{estimate} that
$f$ is radial (take a characteristic function of the ball).
Moreover, it is also clear that the monotonicity condition  is optimal. Indeed, assuming that
$r^{\frac{n-1}{2}} f_0(r) r^{-\varepsilon}$, with a positive $\varepsilon$,
is non-increasing with limit zero,
we observe that the integral in
\eqref{def-fourier}
does not even need to be convergent.
This easily follows from  the well-known asymptotics
\begin{equation}\label{as}
J_{\nu}(y)= \sqrt{\frac{2}{\pi}}\cos(y- \theta_\nu)y^{-\frac{1}{2}}  + B_\nu \sin(y- \theta_\nu)y^{-\frac{3}{2}} + O(y^{-\frac{5}{2}}),
\end{equation}
see \cite[Ch. 7]{watson}.
 \item[($iii$)]
 It is worth mentioning
 that one can alternatively prove Theorem \ref{estimate} with the help of the formula (\ref{as}).
  \item[($iv$)]
Note also that Theorem \ref{estimate} extends the   result by Carton-Lebrun and Heinig  in \cite{car}, which claims that
\eqref{bound for fourier} holds if
$r^{n-1}f_0(r)$ is non-increasing. Moreover, inequality
\eqref{bound for fourier} improves the following
 result
 by
Colzani,  Fontana, and Laeng recently proved in \cite{colzani}:  under conditions of Theorem \ref{estimate}, there holds
$$|\widehat{f}(|\xi|)|\lesssim
\frac1{|\xi|^\frac{n-1}2}
\int_0 ^{|\xi|^{-1}} r^\frac{n-1}2 f_0(r) dr.$$
\end{itemize}
\end{remark}
\begin{proof}
First note that it is enough to show that
$$
|\widehat{f}(1)|\lesssim \int_0 ^1 f_0(r) r^{n-1} dr,$$
as the general result will follow by considering  $x \mapsto f(\lambda x)$ instead of $f$.

 We will make use of following well-known properties of the Bessel functions $J_\alpha(t)$, for $\alpha\ge -\frac12$ and $t>0$,
    \begin{enumerate}
        \item $\int_0 ^t r^{\alpha+1} J_\alpha (r) dr =t^{\alpha+1}J_{\alpha+1}(t);$
        \item $|J_\alpha(t)|\lesssim t^{-\frac{1}{2}};$
        \item $|J_\alpha(t)|\lesssim t^\alpha.$
    \end{enumerate}
In light of property (3) and by integration by parts, we arrive at
   \begin{eqnarray*}
|\widehat{f}(1)|\lesssim\int_{0}^{\infty}r^{n/2-1}f_0(r)J_{n/2-1}(r)r\mathrm{d}r &\lesssim& \int_0^1 r^{n-1}f_0(r) dr + \int_1 ^\infty  r^{\frac{1}{2}}J_{\frac{n}{2}-1}(r)\int_r^\infty   d(f_0(s) s^{\frac{n-1}{2}}) \\
&\lesssim& \int_0 ^1 r^{n-1} f_0(r) dr+ \int_1^\infty d(f_0(s)s^{\frac{n-1}{2}}) \int_1 ^s J_{\frac{n}{2}-1}(r) r^{\frac{1}{2}}
\\
&\lesssim&\int_0 ^1 r^{n-1} f_0(r) dr + f_0(1) \lesssim \int_0 ^1 f_0(r) r^{n-1} dr,
   \end{eqnarray*}
where we have used the following estimate:
given $\alpha\ge -\frac12$, there holds
$$\sup_x \left|\int_0 ^x t^{\frac{1}{2}} J_\alpha(t) dt\right|<\infty.
$$
   Indeed, integrating by parts and using property (1), we have
      \begin{eqnarray*}\int_0 ^x t^{\frac{1}{2}} J_\alpha (t) dt &=& \int_0 ^x t^{\alpha+1} J_\alpha(t) t^{-\frac{1}{2} - \alpha} dt
      = c(\alpha)\int_0 ^x t^{\alpha+1} J_\alpha(t) \int_t^x u^{-\alpha-\frac{3}{2}} du dt\\
      &+& x^{-\alpha-\frac{1}{2}}\int_0 ^x t^{\alpha+1} J_\alpha(t)dt=c(\alpha)\int_0 ^x u^{-\alpha - \frac{3}{2}} u ^{\alpha+1}J_{\alpha+1}(u) du+ x^{\frac{1}{2}}J_{\alpha+1}(x).\end{eqnarray*}
The second term is bounded because of property (2).
For the first one, a further integration by parts yields
   \begin{eqnarray*}\int_0 ^x t^{-\frac{1}{2}} J_{\alpha+1} (t) dt
   &=& c(\alpha)\int_0 ^x t^{\alpha+2} J_{\alpha+1}(t) \int_t^x u^{-\alpha-\frac{7}{2}} du dt+ x^{-\alpha-\frac{5}{2}}\int_0 ^x t^{\alpha+2} J_{\alpha+1}(t)dt\\
   &=&c(\alpha)\int_0 ^x u^{-\frac{3}{2} }J_{\alpha+2}(u) du+ x^{-\frac{1}{2}} J_{\alpha+2}(x),
   \end{eqnarray*}
   which is bounded because of property (3) for small $u$ and (2)  for large $u$.

\end{proof}
Now we are in a position to characterize weighted Fourier inequalities
in the classical $L_p$ spaces for  radial decreasing functions.
\begin{theorem}\label{main}
Let $n\in \N$ and $1\le p< \infty$.
Let $v$ be a positive weight function and $\frac{n-1}{2} \leq \alpha \leq n-1$.
The inequality
\begin{equation}\label{wfi}
\int_{\mathbb{R}^n} |\h{f}(\xi) |^p |\xi|^{n(p-2)} v\big(1/{|\xi|}\big) \,d\xi
 \lesssim
 \int_{\mathbb{R}^n} |{f}(x) |^p v(|x|) \,dx
\end{equation} holds for any  radial non-negative function $f$ such that $f_0(x) x^{\alpha}$ is non-increasing  with limit zero if and only if, for any positive $r,$
\begin{equation}\label{bp}
    r^{p(n-\alpha)}\int_{|x|\geq r} |x|^{-np} v(|x|) d x  \lesssim \int_{|x|\leq r} v(x) |x|^{-\alpha p} dx.
\end{equation}
\end{theorem}
\begin{remark}

\begin{itemize}

\item[($i$)]
Condition \eqref{bp} is monotone with respect to $ \alpha,$ that is, the class of $v$ satisfying \eqref{bp} becomes wider as $ \alpha$ increases.
Thus, Theorem \ref{main} provides a natural balance between the monotonicity  condition on $f$ and the condition on the weight $v.$
For $\alpha=n-1$ Theorem \ref{main} was obtained in \cite{car}.
  \item[($ii$)] Sufficient conditions on weights so that the weighted Fourier inequality \eqref{wfi} holds are usually given in terms of the rearrangement of the weight $v$, cf. \cite{be, he, sin, decarli}. The monotonicity of $f_0$ allows us to obtain necessary and sufficient condition for  \eqref{wfi} in terms of $v$ itself.
\end{itemize}
\end{remark}
\begin{proof}
The proof of the necessity of condition \eqref{bp}  follows by considering
   $f(x)={\chi}(|x|)_{[0,r]} |x|^{-\alpha}$. Then we have  $\widehat{f}(\xi)\gtrsim r^{n-\alpha}$
    for any $|\xi| \leq \frac{\pi}{3r} $.
    Hence,
    $$  r^{p(n-\alpha)}\int_{|\xi|\geq \frac{3}{\pi}r} |\xi|^{-np} v(|\xi|) d \xi
    =
    r^{p(n-\alpha)}\int_{|\xi|\leq \frac{\pi}{3} r^{-1}} |\xi|^{n(p-2)} v(|\xi|^{-1}) d \xi
    \lesssim \int_{|x|\leq r} v(x) |x|^{-\alpha p} dx.$$

{To prove the sufficiency,}
    put $g(t)= f_0(t)t^{\alpha}$ and recall that by hypothesis $g$ is non-increasing. Then,
    Theorem \ref{estimate} yields
     $$
I:=\int_{\mathbb{R}^n} |\h{f}(\xi) |^p |\xi|^{n(p-2)} v(|\xi|^{-1}) \,d\xi \lesssim \int_0 ^\infty \xi ^{n(p-2)+n-1} v(\xi^{-1}) \left(\int_0 ^{1/\xi} g(t) t^{n-1- \alpha} dt  \right)^p d \xi.$$
A change of variables implies that
\begin{eqnarray*}
I
\lesssim
\int_0 ^\infty \xi^{-np+n-1} v(\xi) \left(\int_0 ^{\xi} g(t) t^{n-1- \alpha} dt  \right)^p d \xi
&\lesssim&\int_0 ^\infty g(t)^p t^{-\alpha p + n -1} v(t) dt\\&\asymp&\int_{\mathbb{R}^n} |{f}(x) |^p v(|x|) \,dx,
\end{eqnarray*}
provided that
the following
three weight Hardy inequality
\begin{eqnarray}\label{hardy}
\int_0 ^\infty w(t) 
\Big(H \big(g(t) u(t)\big)  \Big)^p d t
\lesssim \int_0 ^\infty g(t)^p \eta(t) dt,
\end{eqnarray}
with $w(t)=t^{-np+p+n-1}v(t)$,
$\eta(t)=t^{-\alpha p + n -1} v(t)$, and $u(t)=t^{n-1- \alpha}$
holds for any  non-increasing $g$.

The characterization of \eqref{hardy} is known
and, for $p>1$,
according to \cite[Theorem 2.5]{goga}, we need to examine  that,
for any positive $r$, there holds
$$\int_0 ^r
\Big(\int_0 ^t u(s)ds\Big)^p\frac{w(t)}{t^p}dt\lesssim \int_0 ^r \eta(t) dt$$
and
\begin{eqnarray}\label{hardy1} \left(\int_r ^\infty \frac{w(t)}{t^p}dt \right)^{\frac{1}{p}}
\left(\int_0^r \eta(t) \Big(\frac{\int_0^tu(s)ds} {\int_0^t\eta(s)ds} \Big)^{p'}
\right)^{\frac{1}{p'}} \lesssim 1.\end{eqnarray}
The first condition holds since the integrals on the left-hand and right-hand sides  are equivalent.
Regarding the second condition,
observe that, setting $\Delta(t)=\int_0^t\eta(s)ds$,
\begin{eqnarray*}
\int_0 ^r
\Big({\int_0^tu(s)ds}  \Big)^{p'}\frac{\eta(t)}{\Delta^{p'}(t)}
  dt&\asymp& \int_0 ^r  \Big( \int_0 ^t x^{p'(n-\alpha)-1} dx\Big) \frac{\eta(t)}{\Delta^{p'}(t)}  dt \\&=& \int_0 ^r  \Big( \int_x^r\frac{\eta(t)}{\Delta^{p'}(t)}dt\Big) x^{p'(n-\alpha)-1} dx\\
  &\lesssim& \int_0 ^r
{\Delta^{1-p'}(x)}
  x^{p'(n-\alpha)-1} dx.
  \end{eqnarray*}
Further, by our assumption, we have
$$\int_r ^\infty \frac{w(t)}{t^p}dt  =\int_r ^\infty \xi^{-np+n-1} v(\xi) d \xi \lesssim r^{p(\alpha -n)} \int_{0}^r \eta(t) dt
= r^{p(\alpha-n)} \Delta(r).
$$
Therefore, to prove \eqref{hardy1}, it suffices to show that
$$
r^{p(\alpha-n)} \int_0 ^r \eta(t) dt \lesssim \left(\int_0 ^r   \Delta^{1-p'}(x)  x^{p'(n-\alpha)-1} dx\right)^{\frac{1}{1-p'}}.
$$
Now, since $1-p'<0$, the reverse Minkowski inequality (see, e.g., \cite[p. 189]{Bullen}) yields
\begin{eqnarray*}
\left(\int_0 ^r   \Delta^{1-p'}(x) x^{p'(n-\alpha)-1} dx\right)^{\frac{1}{1-p'}}&=&\left( \int_0 ^r  \left(\int_0^x \eta(t) x^{\frac{p'(n-\alpha)-1}{1-p'}} dt \right)^{1-p'}dx\right)^{\frac{1}{1-p'}}
\\
&\geq& \int_0 ^r  \eta (t)\left(\int_t ^r x^{p'(n-\alpha)-1} dx \right)^{\frac{1}{1-p'}} dt\\
&\gtrsim&
\left(r^{p'(n-\alpha)}\right)^{\frac{1}{1-p'}}
\int_0 ^r  \eta(t) dt,
\end{eqnarray*}and the result follows.
In the case $p=1$ the proof is similar with the help of  \cite[Theorem 2.5(d)]{goga}.

\end{proof}

Let us compare our result with the classical Pitt inequality: for $1<  p< \infty$, one has
\begin{eqnarray}\label{pitt}
\int_{\mathbb{R}^n} |\h{f}(\xi) |^p |\xi|^{n(p-2)-\gamma p} \,d\xi
 \lesssim
 \int_{\mathbb{R}^n} |{f}(x) |^p |x|^{\gamma p} \,dx
\end{eqnarray}
for any $f$ provided that $\max(0, n(1-\frac2p))\le\gamma<\frac{n}{p'}$, see e.g. \cite{be}. For $1<p\le 2$ and $\gamma=0$, \eqref{pitt} coincides with  the
Hardy--Littlewood inequality (\ref{hl inequality}).
\begin{corollary}\label{cor-last}
Let $1\le p< \infty$.
    Assume that $f_0(r) r^{\frac{n-1}{2}}$ is non-increasing with limit zero.
Then (\ref{pitt}) holds for any  $\gamma<\frac{n}{p'}$.

\end{corollary}
This result was known for functions $f$ such that
$f_0(r) r^{n-1}$ is non-increasing (see \cite{car}).
Combining various conditions on $f$, we arrive at the following result:
\begin{corollary}\label{cor-last-} Let $1<p<\infty$.
 Inequality
(\ref{pitt})
 holds for any  $f\in X$
if and only if

\begin{itemize}
     \item[($i$)] $\max(0, n(1-\frac2p))\le\gamma<\frac{n}{p'}$ \quad if  \;$X=L_1(\mathbb{R}^n)$, see (\ref{pitt}); 

    \item[($ii$)]  $(n-1)(\frac{1}2- \frac{1}p)+\max(0, 1-\frac2p)\le\gamma<\frac{n}{p'}$ \quad if \; $X$ is the space of radial, integrable functions, see \cite[Th.1.1]{jmaa};

            \item[($iii$)]
            $\frac{n-1}2-\frac np<\gamma<\frac{n}{p'}$ \quad if \; $X$ is the space of radial integrable functions $f$ for which
          $f(x)=f_0(|x|)$ with $f_0\in GM$, that is,
$f_0$ is locally integrable on $\R_+$ and,
for some $\gamma>1$, the inequality
$$\var_{(x,2x)}{f_0}\lesssim \int_{x/\gamma}^{\gamma x}\frac{|f_0(t)|}{t}dt,
\qquad x>0,
$$         holds,
         see \cite{jam};

                 \item[($iv$)] $\gamma<\frac{n}{p'}$  \quad if \; $X$ is the space of radial functions for which
          $f(x)=f_0(|x|)$, with $f_0(r) r^{\frac{n-1}{2}}$ monotonically decreasing on $\R_+$ (see Corollary \ref{cor-last}).

   \end{itemize}
\end{corollary}
It is easy to see the conditions on $\gamma$ become less restrictive from ($i$) to ($iv$).


\vspace{0.5cm}

\begin{thebibliography}{77}















\bibitem[BH]{be}  J. J. Benedetto, H. P. Heinig, Weighted Fourier inequalities: new proofs and generalizations, J.
Fourier Anal. Appl. 9 (2003), no. 1, 1--37.
\bibitem[Bu]{Bullen} P. S. Bullen,
Handbook of Means and Their Applications,
Mathematics and Its Applications, 2003.







\bibitem[CH]{car}
C. Carton-Lebrun; H. P. Heinig,
Weighted Fourier transform inequalities for radially decreasing functions,
SIAM J. Math. Anal. 23 (1992), no. 3, 785--798.

\bibitem[CFL]{colzani}
L. Colzani, L. Fontana, E. Laeng,
Asymptotic decay of Fourier, Laplace and other integral transforms,
J. Math. Anal. Appl. 483, no. 1 (2020), Article  123560.

\bibitem[CDF]{cruz2}
D. Cruz-Uribe, G. Di Fratta, A. Fiorenza,
Modular inequalities for the maximal operator in variable Lebesgue spaces,
Nonlinear Anal., Theory Methods Appl., Ser. A, Theory Methods 177 (2018),  299--311.

\bibitem[CF]{cruz}
D. Cruz-Uribe, A. Fiorenza, Variable Lebesgue Spaces,
Appl. Numer. Harmon. Anal.
Birkh\"{a}user/Springer, Heidelberg, 2013.






\bibitem[De] {alberto1}
A. Debernardi,
The Boas problem on Hankel transforms,
J. Fourier An. Appl, 25 (2019), 3310--3341.

\bibitem[DGT13] {jmaa}
L. De Carli, D.  Gorbachev,
 S.~Tikhonov, Pitt and Boas inequalities for Fourier and Hankel transforms,
 J. Math. Anal. Appl., 408, no. 2, (2013), 762--774.
\bibitem[DGT17] {decarli}
L. De Carli, D.  Gorbachev,
 S.~Tikhonov, { Pitt inequalities and restriction theorems for the Fourier transform,} Rev. Mat. Iberoam. 33, no. 3 (2017),  789--808.
%
%



%
%

\bibitem[GS]{goga}
A. Gogatishvili, V. Stepanov,
Reduction theorems for weighted integral inequalities on the cone of monotone functions,
Russ. Math. Surv. 68, no. 4 (2013), 597--664; translation from Usp. Mat. Nauk 68, no. 4 (2013), 3--68.


\bibitem[GLT]{jam}
D.~Gorbachev, E.~Liflyand, S.~Tikhonov,
Weighted Fourier Inequalities: Boas conjecture in ${\mathbb{R}}^n$,
{J. d'Analyse Math.}, Vol. 114 (2011), 99--120.






\bibitem[He]{he}
H.~Heinig, {Weighted norm inequalities for classes of operators},
Indiana U. Math.~J., 33, no.~4, (1984),  573--582.




\bibitem[Ne]{neu}
C. J. Neugebauer,
Weighted variable $L_p$
integral inequalities for the maximal operator on non-increasing functions,
Stud. Math. 192, no. 1 (2009), 51--60.


\bibitem[RS]{sin}
J. Rastegari, G. Sinnamon, { Weighted Fourier inequalities via rearrangements,}
J. Fourier Anal. Appl. 24, 5 (2018), 1225--1248.

\bibitem[Sa]{sad}
C. Sadosky,
Interpolation of Operators and Singular Integrals. An Introduction to Harmonic Analysis.
 New York, Basel: Marcel Dekker, 1979.

\bibitem[Ti]{tit} E. C. Titchmarsh,
Introduction to the Theory of Fourier Integrals,
Oxford: Clarendon Press, 1937.


\bibitem[Wa]{watson}
G. N. Watson,
A Treatise on the Theory of Bessel Functions. 2nd ed.  Cambridge Univ. Press, 1995.




\end{thebibliography}
\end{document}